\documentclass[12pt,a4paper]{article}

\usepackage[T1]{fontenc}
\usepackage[english]{babel}
\usepackage[utf8]{inputenc}
\usepackage[pdftex]{graphicx}           
\usepackage{setspace}                   
\usepackage{indentfirst}                
\usepackage{makeidx}                    
\usepackage{courier}                    
\usepackage{type1cm}                    
\usepackage{titletoc}
\usepackage[font=small,format=plain,labelfont=bf,up,textfont=it,up]{caption}
\usepackage[usenames,svgnames,dvipsnames]{xcolor}
\usepackage[a4paper,top=3cm,bottom=2.0cm,left=2.54cm,right=2.54cm]{geometry} 
\usepackage[pdftex,plainpages=false,pdfpagelabels,pagebackref,colorlinks=true,citecolor=red,linkcolor=NavyBlue,urlcolor=DarkRed,bookmarksopen=true]{hyperref}
\usepackage[all]{hypcap} 
\fontsize{60}{62}\usefont{OT1}{cmr}{m}{n}{\selectfont}
\usepackage{dsfont}
\usepackage{amsmath}
\usepackage{amsthm,amssymb}
\usepackage{mathabx}

\usepackage{amsfonts}
\usepackage{relsize}
\usepackage{mathrsfs} 
\usepackage{pgf,tikz}
\usepackage{mathabx}
\usetikzlibrary{arrows}
\usepackage{enumitem}
\usepackage{genealogytree}

\usepackage{bbold}

\usepackage{bbm}

\setlist[description]{font=\normalfont\itshape\textbullet\space}
\usepackage{fancyhdr}
\graphicspath{{./figuras/}}             
\makeindex                              
\fontsize{60}{62}\usefont{OT1}{cmr}{m}{n}{\selectfont}
\cleardoublepage
\normalsize

\def\Z{{\mathbb Z}}

\def\1{{\mathbbm 1}}

\def\cR{{\mathcal R}}

\def\B{{\mathcal B}}
\def\C{{\mathcal C}}

\def\cX{{\mathcal X}}
\def\cY{{\mathcal Y}}
\def\cE{{\mathcal E}}
\def\G{{\mathcal G}}
\def\T{{\mathcal T}}
\def\W{{\mathcal W}}

\def\cP{{\mathcal P}}

\def\cN{{\mathcal N}}
\def\F{{\cal F}}

\def\sx{S^{x}}
\def\nx{N^{x}}

\def\ve{\varepsilon}

\def\Pr{{\mathbb P}}

\def\E{{\mathbb E}}

\def\X{\mathcal{X}}

\def\xd{X^{(\d)}}

\def\hxd{\hat X^{(\d)}}

\def\o{\omega}

\def\d{\delta}
\def\a{\alpha}

\def\l{\lambda}

\def\s{\sigma}

\def\vf{\varphi}

\def\ga{\gamma}

\def\={&=&}

\definecolor{title}{RGB}{255,0,90}
\definecolor{def}{rgb}{0.0, 0.28, 0.67}
\definecolor{bege}{rgb}{0.66, 0.637, 0.492}
\definecolor{bpink}{rgb}{0.95,0.82,0.69}
\definecolor{xdxdff}{rgb}{0.49,0.49,1}
\definecolor{qqtttt}{rgb}{0,0.2,0.2}
\definecolor{qqwwcc}{rgb}{0,0.4,0.8}
\definecolor{ffttww}{rgb}{1,0.2,0.4}
\definecolor{qqqqff}{rgb}{0,0,1}
\definecolor{zzttqq}{rgb}{0.6,0.2,0}
\definecolor{ttfftt}{rgb}{0.2,1,0.2}
\definecolor{ttffqq}{rgb}{0.2,1,0}
\definecolor{qqqqff}{rgb}{0,0,1}
\definecolor{qqwwff}{rgb}{0,0.4,1}
\definecolor{ttzzff}{rgb}{0.2,0.6,1}
\definecolor{qqffqq}{rgb}{0,1,0}
\definecolor{fftttt}{rgb}{1,0.2,0.2}
\definecolor{ffqqww}{rgb}{1,0,0.4}
\definecolor{xdxdff}{rgb}{0.49,0.49,1}
\definecolor{cqcqcq}{rgb}{0.75,0.75,0.75}

\newtheorem{teorema}{Theorem}
\newtheorem{lema}{Lemma}[section]
\newtheorem{Lema}{Lemma}[subsection]

\newtheorem{observacao}[teorema]{Remark}
\newtheorem{Remark}[Lema]{Remark}
\newtheorem{proposicao}[teorema]{Proposition}

\author{Luiz Renato Fontes \footnote{IME-USP, Rua do Mat\~ao 1010, 05508-090 S\~ao Paulo SP,  Brazil, lrfontes@usp.br} \thanks{Supported in part by CNPq grants 309938/2025-2 and 408590/2024-6, and FAPESP grant 2023/13453-5}
	\and Pablo A.~Gomes \footnote{IME-USP, Rua do Mat\~ao 1010, 05508-090 S\~ao Paulo SP,  Brazil, pagomes@usp.br} \thanks{Supported in part by FAPESP  grants 2023/13453-5 and 2025/27064-6} }
\title{Random walk in a null recurrent birth-and-death dynamical environment} 

\date{}

\begin{document}
	
	\maketitle


    	\begin{abstract}

		We study a random walk in $\Z^d$ in a dynamical random environment constituted of independent 
		null recurrent birth-and-death chains, and establish a scaling limit result for the position of the walker
		in terms of a sublinear time change of Brownian motion, thus establishing subdiffusivity, in contrast to a
		diffusivity result established previously in an ergodic version of the model.

	\end{abstract}


	\smallskip



\section{Introduction}
\label{intro}

This article may be seen as  a report on a followup investigation on the continuous time random walk in dynamical random environment studied in~\cite{FGP}.
As in that paper, we have an embedded (discrete time) mean zero random walk on $\Z^d$, which we presently take to be simple and symmetric 
for simplicity. The environment is also given by a family of independent and equally distributed birth-and-death processes indexed by $\Z^d$, but now,
in contrast to~\cite{FGP}, they are not ergodic; for simplicity, we consider spatially homogeneous transitions only (off the origin) for the environment.
The jump rate of the random walk at space time point $(x,t)$ is again given by a decreasing function of the environment at $(x,t)$, which we now take to have a specific form, with a specific decay, namely, (in the notation of~\cite{FGP}), $\varphi(x)=(x+1)^{-\a}$, with $\a\in(0,2)$ for null recurrent environments\footnote{which is the main case we study here, but we do discuss the (simpler) transient case of the environment, briefly at the end of the paper, where we have $\a\in(0,1)$ instead}.

Our main result, Theorem~\ref{main}, stated in the next section for the null recurrent case, is the scaling limit of the suitably rescaled position of the walker; the limit is a time changed Brownian Motion, with the time change a polynomial function of time, with degree less than one, amounting to a subdiffusive process, in contrast to the usual/diffusive CLT obtained in~\cite{FGP}, under 'strong' ergodicity (according to the tentative terminology adopted in that paper). Our strategy requires that the dimension be at least five. There is no dimension restriction in the results
of~\cite{FGP}, as well as, as briefly discussed at the end of the paper, for the transient version of
our environment.

A natural motivation to our study of random walks in birth-and-death dynamical environments is to provide a dynamical version of Bouchaud’s trap model, originally introduced in~\cite{B92}, where the model describes a continuous-time random walk on $\mathbb{Z}^d$, whose embedded chain is also an ordinary random walk, but whose jump rates are spatially inhomogeneous and determined by an iid random field representing traps. In the regime where the inverses of the jump rates have heavy-tailed distributions, the random walk exhibits subdiffusive behavior, as well as the phenomenon of aging, as analyzed in~\cite{BAC} and~\cite{FIN02}. The aim of our main result described above, is to show how a non-ergodic dynamical environment can lead to subdiffusive behavior of the random walk. Due to the non-ergodic nature of the environment, the walker may remain influenced by traps over certain time scales. This motivates an appropriate rescaling of time, which allows us to better understand the effect of the traps on the behavior of the random walk.

The literature on random walks in random environments is quite large, with many papers dealing with several classes of Markovian environments (see, for example, \cite{BZ06, BMP97, BMPZ, DKL08, RV13}). There is also a literature that considers interacting particle systems as the environment and jump rates governed by the configuration in a neighborhood of the walker (e.g., \cite{dHdS, MV, BHT, HKT}). Many of these works allow for correlations in the environment, going well beyond our independence assumption, and often assuming an ellipticity condition on the jump rates, which we do {\em not} require. For a more extended list references and a discussion on the literature see~\cite{FGP}. We are not aware of studies on similar models in Markovian null recurrent or transient environments. 

This paper is organized as follows. In the next section we set up the model in detail and formulate our main result. In the following section, we argue that a proof of that result follows readily from a law of large numbers on the suitably scaled sequence of jump times of the random walk.
The bulk of the paper is the proof of the law of large numbers, whose strategy and overall ideas are presented at the beginning of Section~\ref{proof-prop}, with detailed arguments in the following subsections.



\renewcommand{\thefootnote}{\fnsymbol{footnote}}

\section{Model and result}
\label{mores}

\setcounter{equation}{0}

Let $H:=\{(\eta^x_t)_{t\geq0},\,x\in\Z^d\}$ be an independent family of symmetric, rate 1, birth-and-death processes starting from the origin. For $x\in\Z^d$, we write $\eta^x_t$ as
\begin{equation}\label{rep}
	\eta^x_t=\big|\sx_{\nx_t}\big|,\,t\geq0,
\end{equation}
where $\{(\sx_n)_{n\geq0},\,x\in\Z^d\}$ is a family of independent simple symmetric random walks in $\Z$, 
and $\{(\nx_t)_{t\geq0}, \,x\in\Z^d\}$ is a family of independent rate 1 Poisson counting processes, independent of the random walk family.

Given $H$ as above, consider a Markov jump process $X=(X_t)_{t\geq0}$ starting at the origin, whose jump chain $(\X_n)_{n\geq0}$ is a  simple symmetric random walk in $\Z^d$, and whose (space-time inhomogeneous) jump rates are given by
\begin{equation}\label{la}
	\l^x_t=\vf(\eta^x_t)=(\eta^x_t+1)^{-\a},
\end{equation}
with $\a\in(0,2)$.

\begin{teorema}\label{main}
	Let $d\geq5$, and, for $\d>0$, set $\xd=(\xd_t)_{t\geq0}$, where 
	\begin{equation}\label{xd}
		\xd_t=\d^{\frac{2-\a}{4}}X_{\d^{-1}t},\, t\geq0.
	\end{equation}
	Then for a.e.~realization of the environment
	\begin{equation}
		\xd\Rightarrow
		\big(B_{t^{1-\frac{\a}2}}\big)_{t\geq0} \text{ as } \d\to0,
	\end{equation}
    in the finite dimensional sense, 
	where 
	$(B_t)_{t\geq0}$ is Brownian motion with diffusion coefficient $1/\mu$, with $\mu\in(0,\infty)$ to be specified below ---
	see~\eqref{c15}. 
\end{teorema}

\begin{observacao}\label{fidi}
	The above convergence ought to hold in a stronger sense as well, namely, in Skorohod space, at least, conceivably, at high dimensions, but we avoid a stronger statement in order to not go too deep into technicalities involved in verifying standard 
	criteria for such kinds of convergence, which, we expect, would mostly involve standard arguments.
\end{observacao}



\section{Proof of Theorem~\ref{main}}
\label{proofs}

\setcounter{equation}{0}

Our setup allows for a time change which makes the environment of  "order 1", as follows.
Let $W^x_0=0$, $W_x(t)=(\eta^x_t+1)/\sqrt t$, $t>0$, $x\in\Z^d$, and set
\begin{equation}\label{y}
	Y_t=X_{t^\ga},\,t\geq0,
\end{equation}
where $\ga=\frac2{2-\a}$.

\begin{lema}\label{version}
	Given $H$, $(Y_t)_{t\geq0}$ is a Markov jump process starting at the origin, with jump chain $(\X_\ell)_{\ell\geq0}$, and instantaneous jump rates at $(x,t)$ given by $\rho^x_t:=\ga (W^x_{t^\ga})^{-\a}$, $t\geq0$, $x\in\Z^d$.
\end{lema}

\begin{proof}[\textbf{Proof}] This is a straightforward argument, which we outline for completeness. 

That $(Y_t)$ inherits the Markovian character and jump chain from $(X_t)$ is quite clear. 
It remains to characterize the jump times of $(Y_t)$, which we will relate to those of $(X_t)$, as follows. The latter may be obtained,
via  a standard Harris graphical construction, from the family $\{ \cP^x,\,x\in\Z^d\}$ of independent Poisson point processes
on $[0,\infty)$ with intensity function $\l^x_\cdot$, respectively. 

It is now enough to check that for each $x\in\Z^d$, $\cP^x$ is mapped, by the transformation $t\to t^\ga$, to $\tilde\cP^x$, a  Poisson point processes
on $[0,\infty)$ with intensity function $\ga (W^x_{t^\ga})^{-\a}$, $t\geq0$.

That such a transformation maps a Poisson point process to a Poisson point process is quite clear, so it is enough to obtain 
the expected value of the occupation random variables of $\tilde\cP^x$, namely, $\E(\#(r,s))$, $0\leq r<s<\infty$, where
\begin{equation}\label{occ}
	\#(r,s)=\big|[r,s]\cap\tilde\cP^x\big|,
\end{equation}
where $|\cdot|$ means cardinality of $\cdot$. 

But the right hand side of~\eqref{occ} equals $\big|[r^\ga,s^\ga]\cap\cP^x\big|$, and 
\begin{equation}\label{eocc}
	\E\big\{\big|[r^\ga,s^\ga]\cap\cP^x\big|\big\}=\int_{r^\ga}^{s^\ga}\l^x_t dt.
\end{equation} 
We conclude by resorting to an obvious change of variables in the  integral to find that it equals $\int_{r}^{s} \rho^x_t\,dt$.
\end{proof}

\bigskip

Now, given $(\X_\ell)_{\ell\geq0}$ and $\{ \tilde\cP^x,\,x\in\Z^d\}$, we obtain the successive jump times of $(Y_t)_{t\geq0}$ as follows.
Set $\s_0=0$, and for $\ell\geq1$, given that $\X_{n-1}=x$, let
\begin{equation}\label{sn}
	\s_\ell=\inf\{t>\s_{\ell-1}:\,t\in\tilde\cP^x\}
\end{equation}
be the smallest point of $\tilde\cP^x$ greater than $\s_{\ell-1}$. We may then write

\medskip

\begin{equation}\label{xtg}
	Y_{t}=\sum_{\ell\geq0}\X_\ell\1\big\{t\in[\s_\ell,\s_{\ell+1})\big\}
	=\sum_{i=1}^{M_t}\xi_i,
\end{equation}
where $(M_t)_{t\geq0}$ is the counting process associated to $(\s_\ell)_{\ell\geq1}$, namely,
\begin{equation}\label{mt}
	M_t=\sup\{\ell\geq0:\,\s_\ell\leq t\},\,t\geq0,
\end{equation}
and
$\xi_i,\,i\geq1,$ represent the iid jumps of $(\X_\ell)_{\ell\geq0}$ (so that 
$\X_\ell=\sum_{i=1}^\ell\xi_i$, $\ell\geq1$).

\medskip

The proof of Theorem~\ref{main} follows readily 
from the next result.

\begin{proposicao}\label{lln}
	If $d\geq5$, then for a.e.~realization of the environment,
	\begin{equation}\label{eq:lln}
		\frac{\s_\ell}{\ell}\to\mu
	\end{equation}
	in probability as $\ell\to\infty$, where $\mu\in(0,\infty)$ will be specified below ---
	see~\eqref{c15}. 
\end{proposicao}




\section{Proof of Proposition~\ref{lln}}
\label{proof-prop}

\setcounter{equation}{0}

For $n\geq0$, let us consider the range of $(\X_\ell)$\footnote{We will drop subindices from the notation of processes from now on; so in this case,
	we denote $(\X_n)_{n\geq0}$ simply by $(\X_n)$.} after $n$ steps, namely 
\begin{equation}\label{range}
	\cR_n=\{x\in\Z^d:\, \X_\ell=x\,\text{ for some }\,\ell\leq n\},
\end{equation}
and consider the sequence of {\em novelty times}, defined as follows. Let $\tau_0=0$ and for $n\geq1$
\begin{equation}\label{kappa}
	\tau_n=\inf\{\ell>\tau_{n-1}:\, \X_\ell\notin\cR_{n-1}\},
\end{equation}
so that 
\begin{equation}\label{yn}
	\cY_n:=\X_{\tau_n},\,n\geq0, 
\end{equation}
are the {\em new} sites visited by $(\X_\ell)$ in (chronological) order.
Let us next consider the {\em last visit times} to each of these  {\em new} sites, namely, 
\begin{equation}\label{omega}
	\o_n=\sup\{i\geq\tau_n:\, \X_j\ne\cY_n\,\text{ for all }\,j>i\}.
\end{equation}

\begin{observacao}\label{bc}
	Let
	\begin{equation}
		\kappa_n=\o_n-\tau_n,\,n\geq0.
	\end{equation}
	By Markovianity, $\kappa_n\sim\o_0$. The latter random variable, the last visit of $(\X_n)$ to the origin is, in $d\geq3$, well known to be distributed as follows:
	\begin{equation}\label{lastail}
		\Pr(\o_0=\ell)=\text{const }\Pr(\X_\ell=0),\,\ell\geq0,
	\end{equation}
	and thus, for $n\geq0$,
	\begin{equation}\label{kapadis}
		\Pr(\kappa_n=\ell)\sim\text{const }\ell^{-d/2}.
	\end{equation}

\end{observacao}

The strategy in the remainder of the proof is to argue that for all large times, the jump rate at the current position of $(X_t)$ is roughly the same 
as at the beginning of the first visit to that site, which is in turn roughly distributed as $\ga|\cN|^{-\a}$, where $\cN$ is a standard Gaussian random variable, 
independently from site to site. This roughly implies that the $\s_n$ may be roughly written as a sum 
over $\cR_n$ of products of iid random variables distributed as $\ga^{-1}|\cN|^{\a}$
and sums of a geometric number\footnote{counting the total number of visits to each site of $\cR_n$}  of iid standard exponential random variables. The well known behavior of $|\cR_n|$, combined with the Ergodic Theorem, would then yield the result.

\begin{observacao}\label{approx}
	The upshot is that $(Y_t)$ behaves (roughly) as a continuous time random walk in a {\em static} random field of iid jump rates with marginals distributed as $\ga|\cN|^{-\a}$. 
\end{observacao}

The work to control the approximations involved in this strategy relies on a fast enough decay of  the distribution of $\kappa_n$, which requires high dimensions, and on standard moderate deviation results for $(S^\cdot_n)$, combined with the slow decay of $\vf$.
There are two main ingredients for that: linear growth of arrival times at new sites, and sublinear bound on successive times between first and last visits to new sites; we devote the next subsections for that.



\subsection{Linear growth of arrival times at new sites}
\label{lingro}

For $n\geq0$, let $\hat\tau_n$ the time of first arrival by $(Y_t)$ at $\cY_n$, that is,
\begin{equation}\label{tau}
	\hat\tau_n=\inf\{t\geq0:\,Y_t=\cY_n\}.
\end{equation}

\begin{Remark}\label{taufin}
	It is perhaps quite clear that $\hat\tau_n$ is almost surely finite for every $n$, but let us give next a brief argument. 
	
	It is enough to show that jump times are all finite almost surely. Let us fix $t\geq0$ and $x\in\Z^d$. By Markovianess, given that $Y_t=x$, the probability that  the next jump time is infinite is given by 
	\begin{equation}\label{f1}
		\E\Big[\exp\big\{-\ga\int_{t}^\infty (W^x_{s^\ga})^{-\a}ds\big\}\Big].
	\end{equation}
	
	It suffices to show that the above  integral diverges in probability. Given $L>1$, we have that
	\begin{equation}\label{f2}
		\int_{t}^\infty (W^x_{s^\ga})^{-\a}ds\geq\int_{t+L}^{t+2L} (W^x_{s^\ga})^{-\a}ds\geq (M_{t,L})^{-\a} L,
	\end{equation}
	where $M_{t,L}=\max\big\{W^x_s:\,(t+L)^\ga\leq s\leq (t+2L)^\ga\big\}$.
	It is enough then to argue that  $M_{t,L},\,t>0,\,L>1$, are tight.
	We have that
	\begin{equation}\label{f3}
		M_{t,L}\leq1+(t+L)^{-\ga/2}
		\max_{0\leq s\leq (t+2L)^\ga}\big|S^x_{N^x_s}\big|.
	\end{equation}
	
	Kolmogorov's Maximal Inequality implies that the probability that the latter term on the right hand side exceeds an arbitrary $K>0$ is bounded above by 
	$2^\ga K^{-2}$, and tightness follows.

\end{Remark}

\begin{Lema}\label{lg}
	There exists constants $a,b>0$ such that for all large $n$
	\begin{equation}\label{exp}
		\Pr(\hat\tau_n\leq an)\leq e^{-bn}.
	\end{equation} 
\end{Lema}

\begin{proof}[\textbf{Proof.}]
Let $\G_t=\s(Y_s,\eta^x_s;\,s\geq0,\,x\in\Z^d)$, $t\geq0$, be the (joint) filtration generated by $(Y_t)$ and $(\eta^x_t)$, and set $\F_n=\G_{\hat\tau_n}$, $n\geq0$.   We have that
\begin{equation}\label{ts1}
	\hat\tau_n\geq\sum_{i=0}^{n-1} \cE_i\geq\sum_{i=0}^{n-1} \1_{\{\cE_i>1\}}=:B_{n-1},\,n\geq1,
\end{equation}
where $\cE_i$ denotes the time elapsed since $\hat\tau_i$ until the next jump of $(Y_t)$. 
Notice that $\cE_i\in\F_n$ if $i<n$.
By a standard argument using an exponential Chebyshev inequality, it is enough to show that there exists $\zeta<1$ such that
\begin{equation}\label{lapb}
	\E(e^{- B_n})\leq\zeta^n,\,n\geq1.
\end{equation}

We next argue that~\eqref{lapb} follows from the following claim:
\begin{equation}\label{zeta}
	p:=\inf_{i,t\geq0,,x\in\Z^d}\Pr\big(\cE_i>1\big|\hat\tau_i=t,\cY_i=y\big)>0.
\end{equation}
Indeed, let us write
\begin{equation}\label{lapb1}
	\E(e^{- B_n}|\F_n)=e^{- B_{n-1}} \E(e^{-\1_{\{\cE_n>1\}}}|\F_n)=e^{- B_{n-1}} \big[1-(1-e^{-1})\Pr(\cE_n>1|\F_n)\big],
\end{equation}
$n\geq0$ (where $B_{-1}:=0$).
By Markovianess and spatial homogeneity,  the expression in square brackets is bounded above by $1-(1-e^{-1})p=:\zeta<1$. Applying this bound to the right hand side of~\eqref{lapb1}, and iterating, we establish~\eqref{lapb}.

In order to check~\eqref{lapb1}, let us first notice that, again by Markovianess and spatial homogeneity, 
\begin{equation}\label{ts2}
	p=\inf_{t\geq0}\,\E\Big[\exp\big\{-\ga\int_{t}^{t+1} (W^0_{s^\ga})^{-\a}ds\big\}\Big].
\end{equation}

If $t<1$, then the latter expectation is bounded  below by
\begin{equation}\label{ts3}
	\E\big[e^{-\ga'(m+1)^{-\a}}\big] 
\end{equation}
where $\ga'=\ga\int_{1}^{2} s^{\a/(2-\a)}\,ds=2^\ga-1$, and $m=\min_{0\leq s\leq2^\ga}|S^0_{N_s}|$; it readily follows that the expression in~\eqref{ts3} is a (strictly) positive constant.

If $t\geq1$, we 
find that the expectation on the right hand side of~\eqref{ts2} is bounded  below by
\begin{equation}\label{ts4}
	\E\big[e^{-\ga \,m_t^{-\a}}\big],
\end{equation}
where $m_t=\min\big\{W^0_s,\,t^\ga\leq s\leq(t+1)^\ga\big\}$, and to find an uniformly positive lower bound for~\eqref{ts4}, it is enough to argue that $\tilde m_t$ converges in distribution as $t\to\infty$ to a nontrivial (positive) random variable.
We have that
\begin{equation}\label{ts5}
	t^{-\ga/2}\min_{t^\ga\leq s\leq (t+1)^\ga}\big|S^0_{N^0_s}\big|
	\leq m_{t}
	\leq t^{-\ga/2}+t^{-\ga/2}\,\big|S^0_{N^0_{t^\ga}}\big|.
\end{equation}
The left hand side of~\eqref{ts5} is bounded below by 
\begin{equation}\label{ts6}
	t^{-\ga/2}\,\big|S^0_{N^0_{t^\ga}}\big| -t^{-\ga/2}\max_{t^\ga\leq s\leq (t+1)^\ga}\big|S^0_{N^0_s}-S^0_{N^0_{t^\ga}}\big|,
\end{equation}
and the CLT and Kolmogorov's Maximal Inequality imply that $m_{t}$ converges in distribution as $t\to\infty$ to $|\cN|^{-\a}$, 
where $\cN$ is a standard Gaussian random variable.
\end{proof}

\subsection{Sublinear bound on successive times between first and last visits to new sites}
\label{slingro}

For $n\geq0$, let $\hat\o_n$ the (end) time of the last visit by $(Y_t)$ at $\cY_n$, that is,
\begin{equation}\label{hom}
	\hat\tau_n=\sup\{t\geq0:\,Y_t=\cY_n\}.
\end{equation}

\begin{Lema}\label{slg}
	Let $\hat\kappa_n=\hat\o_n-\hat\tau_n$, $n\geq0$. If $d\geq5$, we may choose a constant $\d'<1$ such that 
	\begin{equation}\label{subexp}
		\sum_{n\geq0}\Pr(\hat\kappa_n> \hat\tau_n^{\d'})<\infty. 
	\end{equation} 	
\end{Lema}

\begin{proof}[\textbf{Proof}]
We write
\begin{equation}\label{se1}
	\hat\kappa_n=\sum_{j=1}^{\kappa_n}\cE^n_j,
\end{equation} 	
where $\cE^n_1$ is the duration of the visit by $(Y_t)$ to $\cX_{\tau_n}$ starting at $\T^n_0:=\hat\tau_n$, and, inductively, for $j\geq2$,
$\cE^n_j$  is the duration of the visit by $(Y_t)$ to $\cX_{\tau_n+j-1}$ starting at $\T^n_{j-1}:=\hat\tau_n+\sum_{i=1}^{j-1}\cE^n_i$. 

By~\eqref{kapadis}, we may assume that $\kappa_n\leq n^\d$ for $\d\in\big(\frac2{d-2},1\big)$ fixed.
We may also assume, by Lemma~\ref{lg}, that $\tau_n>an$. 

Let us fix $c\in(0,1-\d)$. By Markovianess and spatial homogeneity, 
\begin{align}\label{se2}
	\Pr\big(\cE^n_1>t_0^c\big|\hat\tau_n=\T^n_{0}=t_0,Y_{t_0}=y\big) &=\E\Big[\exp\big\{-\ga\int_{t_0}^{t_0+t_0^c} (W^0_{s^\ga})^{-\a}ds\big\}\Big]
	\nonumber \\ &\leq \E\big(e^{-\ga t_0^c(M_{t_0})^{-\a}}\big),
\end{align}
where 
$M_t=\max\big\{W^0_s:\,t^\ga\leq s\leq (t+t_0^{c+\d})^\ga\big\}$, $t\geq0$.

Let us set $c'=c/(2+\a)$. Given $t_0>an$, we have that
\begin{align}\label{se3}
	\Pr\big(M_{t_0}>t_0^{c'}\big)&\leq\Pr\Big(\max_{t_0^\ga\leq s\leq (t_0+t_0^{c+\d})^\ga}|S^0_{N^0_s}|>t_0^{\ga/2} t_0^{c'}-1\Big)
	\nonumber \\ &= \Pr\Big(\max_{N^0_{t_0^\ga}\leq k\leq N^0_{(t_0+t_0^{c+\d})^\ga}}|S^0_{k}| >t_0^{\ga/2} t_0^{c'}-1\Big).
\end{align}
Resorting to standard large deviation bounds for $(N^0_\cdot)$, it is enough to bound 
\begin{equation}\label{se4}
	\Pr\Big(\max_{\frac12t_0^\ga\leq k\leq \frac32t_0^\ga}|S^0_{k}|>t_0^{\ga/2} t_0^{c'}-1\Big)\leq e^{-c''t_0^{2c'}},
\end{equation}
where $c''>0$ is a constant, and the latter inequality follows from the union bound and  Bernstein Inequality. It follows that
\begin{equation}\label{se5}
	\Pr\big(\cE^n_1>t_0^c\big|\T^n_{0}=t_0,Y_{t_0}=y\big)\leq e^{-\ga t_0^{c-\a c'}}+e^{-c''t_0^{2c'}}\leq 2e^{-\ga''t_0^{2c'}},
\end{equation}
where $\ga''=\ga\wedge c''$. 

Proceeding inductively, we may repeat the latter approach to find that, for $1\leq j<\lfloor n^\d\rfloor$,
we have that, for $t_0\leq t_j\leq t_0+jt_0^c$,
\begin{align}\label{se6}
	\Pr\big(\cE^n_{j+1}>t_0^c\big|\hat\tau_n=t_0,\T^n_{j}=t_j,Y_{t_j}=y\big) &= \E\Big[\exp\big\{-\ga\int_{t_j}^{t_j+t_0^c} (W^0_{s^\ga})^{-\a}ds\big\}\Big]
	\nonumber \\ 
    &\leq \E\big(e^{-\ga t_0^c(M_{t_j})^{-\a}}\big),
\end{align}
and the bound on the right hand side of~\eqref{se5} follows for the left hand side of~\eqref{se6} with a similar argument as above.

It follows that outside of an event of probability bounded above by $2n^\d e^{-c''n^{2ac'}}$ (on account of the above
estimates on  $\cE^n_j$, $1\leq j\leq n^\d$), combined with the estimates of the probability of $\{\kappa_n>n^\d\}$, from~\eqref{kapadis},
and of  $\{\tau_n\leq an\}$, from Lemma~\ref{lg}, we have that, outside of an event with summable probability,
\begin{equation}\label{se7}
	\hat\kappa_n\leq n^{\d+c},
\end{equation} 	
and the result follows, using again Lemma~\ref{lg}, for any choice $\d'\in(\d+c,1)$.
\end{proof}

\subsection{Conclusion of proof of Proposition~\ref{lln}}
\label{conc}

We will fix $(\cX_n)$ and resort to a convenient construction of $\{(\eta^{\cY_n}_t),\,n\geq0\}$, which is the relevant environment,
as follows. By Skorohod representation, we may write $(S^{\cY_n}_\ell)$ for each $n\geq0$ as
\begin{equation}\label{c1}
	S^{\cY_n}_\ell=\B_n(T^n_\ell),\,\ell\geq0,
\end{equation} 
where $\B_n$, $n\geq0$, are independent standard Brownian motions, and $(T^n_\ell)$, $n\geq0$, are iid sequences of
iid non negative random variables with mean 1 and a positive exponential moment. 

It is also quite clear that, for $n\geq0$, the relevant environment in $\cY_n$ is the one between times $\hat\tau_n^\ga$ and
$(\hat\tau_n+\hat\kappa_n)^\ga$, namely
\begin{equation}\label{c2}
	(\eta^{\cY_n}_t)\big|_{\hat\tau_n^\ga}^{(\hat\tau_n+\hat\kappa_n)^\ga}
	=\left.\left(\frac{|\B_n(T^n_{N^{n}_t})|+1}{\sqrt t}\right)\right|_{\hat\tau_n^\ga}^{(\hat\tau_n+\hat\kappa_n)^\ga}
	=\left.\left(\Big|\frac{\B_n(T^n_{N^{n}_t})}{\sqrt t}\Big|+\frac1{\sqrt t}\right)\right|_{\hat\tau_n^\ga}^{(\hat\tau_n+\hat\kappa_n)^\ga},
\end{equation}
where $(N^n_t)=(N^{\cY_n}_t)$, $n\geq0$.
We  now write the the first term within parentheses on the right hand side, for $n\geq1$, as
\begin{equation}\label{c3}
	\Big|\frac{\B_n(T^n_{N^{n}_t})}{\hat\tau_n^{\ga/2}}\Big|\times  \sqrt{\frac{\hat\tau_n^{\ga}}t}    
	=\big|\tilde \B_n(\tilde T^n_{t})\big|\times \sqrt{\frac{\hat\tau_n^{\ga}}t},
\end{equation}
where $\tilde T^n_{t}=T^n_{N^{n}_t}/\hat\tau_n^\ga$, $t\geq\hat\tau_n^\ga$, and $\tilde \B_n$, $n\geq1$, are independent standard Brownian motions. The latter equality follows from well known scale invariance of Brownian motion, once we observe that $\hat\tau_n$ and $\B_n$ are independent.

Now, for $n\geq0$, let 
\begin{eqnarray*}
	W^-_n\=\min\big\{W^{\cY_n}_t:\,\hat\tau_n^\ga\leq t\leq(\hat\tau_n+\hat\kappa_n)^\ga\big\},\\
	W^+_n\=\max\big\{W^{\cY_n}_t:\,\hat\tau_n^\ga\leq t\leq(\hat\tau_n+\hat\kappa_n)^\ga\big\};
\end{eqnarray*}
and, given $\ve\in(0,1)$, set
\begin{eqnarray*}
	V^-_n\,=\,V^-_n(\ve)\=(1-\ve)\min\!\big\{|\tilde \B_n(t)|:\,1-\ve\leq t\leq1+\ve\big\},\\
	V^+_n\,=\,V^+_n(\ve)\=\ve+\max\!\big\{|\tilde \B_n(t)|:\,1-\ve\leq t\leq1+\ve\big\}.
\end{eqnarray*}

From standard large deviation estimates for Poisson processes and sums of iid random variables with a positive exponential moments,
as well as Lemmas~\ref{lg} and~\ref{slg} above, we have that, except for an event whose probability is summable in $n$,
\begin{equation}\label{c4}
	V^-_n\leq W^-_n\leq W^+_n\leq V^+_n,
\end{equation}
and thus there exists an almost surely finite random variable $\tilde N$ such that~\eqref{c4} holds for all $n\geq\tilde N$.

Let us at this point resort to the following construction of the jump times of $(Y_t)$. Let $Z_1,Z_2,\ldots$ be a sequence of independent standard exponential random variables, independent from $(\cX_\ell)$ and $H$.
Recall  the notation for the successive jump times  $Y_t$: $\s_0=0,\s_1,\s_2,\ldots$. Then for $\ell\geq1$, given that $\s_{\ell-1}=s$ and $Y_{\s_{\ell-1}}=y$,
\begin{equation}\label{c5}
	\s_\ell-s=\W^{-1}_{y,s}(Z_\ell),
\end{equation}
where $\W^{-1}_{y,s}(\cdot)$ is the inverse of $\W_{y,s}(\cdot)=\ga\int_s^{s+\cdot}(W^y_{t^\ga})^{-\a}\,dt$. 

It follows that  there exists an almost surely finite random variable $\tilde L$ such that~\eqref{c4} holds for all $\ell\geq\tilde L$,
and thus for  $\ell\geq\tilde L$ we have that
\begin{equation}\label{c6}
	U_{n_\ell}^-\,Z_\ell:=\ga^{-1}(V^-_{n_\ell})^{\a}\, Z_\ell\leq	\s_\ell-\s_{\ell-1}\leq \ga^{-1}(V^+_{n_\ell})^{\a}\, Z_\ell=:U_{n_\ell}^+\,Z_\ell,
\end{equation}
where $n_\ell$ is such that $Y_{\s_\ell}=\cY_{n_\ell}$, so it is enough to show that
\begin{equation}\label{c7}
	\liminf_{\ell\to\infty}\frac1\ell\sum_{i=1}^\ell U_{n_\ell}^-\,Z_\ell\geq\C_\ve^-;\quad
	\limsup_{\ell\to\infty}\frac1\ell\sum_{i=1}^\ell U_{n_\ell}^+\,Z_\ell\leq\C_\ve^+;
\end{equation}
in probabilty, where $\C_\ve^-,\,\C_\ve^+$ are deterministic and $\lim_{\ve\to0}\C_\ve^-=\lim_{\ve\to0}\C_\ve^+$ is nontrivial. Let us consider the latter sum. It may be written as 
\begin{equation}\label{c8}
	\sum_{i=1}^{R_\ell} U_{i}^+\sum_{j=1}^{G_{i,\ell}}\,\tilde Z^i_j\leq \sum_{i=1}^{R_\ell} U_{i}^+\sum_{j=1}^{G_{i}}\,\tilde Z^i_j,
\end{equation}
where $R_\ell=|\cR_\ell|$ is the number of new sites visited by $(\cX_n)$ in the first $\ell$ steps, $G_{i,\ell}$ is the number of visits by
$(\cX_n)$ to $\cY_i$ in the first $\ell$ steps, $G_i=\lim_{\ell\to\infty}G_{i,\ell}$, and $\tilde Z^i_j$, $i,,j\geq1$, are indepedent exponential random variables. 

We note that $G_1,G_2,\ldots$ is a stationary ergodic sequence, and $G_1$ is a Geometric random variable with success parameter $p_d=\Pr(\cX_n\ne0,\,n\geq1)$ and mean $p_d^{-1}$. Also, it is a well known fact that 
\begin{equation}\label{c9}
	\lim_{\ell\to\infty}\frac{R_\ell}\ell \to p_d
\end{equation}
almost surely. From this and the Ergodic Theorem, using also the fact that $U^+_1, U^+_2,\ldots$ are iid and integrable, we find that the limsup in~\eqref{c7} is bounded above by
\begin{equation}\label{c10}
	\lim_{\ell\to\infty}\frac1\ell\sum_{i=1}^{R_\ell} U_{i}^+\sum_{j=1}^{G_{i}}\,\tilde Z^i_j=\E(U^+_1)=:\C_\ve^+
\end{equation}
almost surely.

To deal with the liminf in~\eqref{c7}, we write the respective sum as 
\begin{equation}\label{c11}
	\sum_{i=1}^{R_\ell} U_{i}^-\sum_{j=1}^{G_{i,\ell}}\,\tilde Z^i_j 
	\geq \sum_{i=1}^{R_{(1-\ve)\ell}} U_{i}^-\sum_{j=1}^{G_{i,\ell}}\,\tilde Z^i_j
\end{equation}

Let now $A_{\ell,\ve}=\big\{G_{i,\ell}\ne G_i$ for some $i\leq R_{(1-\ve)\ell}\big\}$. 
Resorting to a straighforward union bound, we find that
\begin{equation}\label{c12}
	\Pr(A_{\ell,\ve})\leq\ell\,\Pr(\o_0>\ve\ell)\leq \text{const }\ell^{-(d-4)/2}\to0,\,\text{ as }\ell\to\infty,
\end{equation}
where we have used~\eqref{kapadis} in the latter inequality.

Since the right hand side of~\eqref{c11} equals $\sum_{i=1}^{R_{(1-\ve)\ell}} U_{i}^-\sum_{j=1}^{G_{i}}\,\tilde Z^i_j$ in $A_{\ell,\ve}$, we conclude, arguing as for~\eqref{c10}, that the liminf in~\eqref{c7} is bounded below by
\begin{equation}\label{c13}
	\lim_{\ell\to\infty}\frac1\ell\sum_{i=1}^{R_{(1-\ve)\ell}} U_{i}^-\sum_{j=1}^{G_{i}}\,\tilde Z^i_j=(1-\ve)\E(U^-_1)=:\C_\ve^-
\end{equation}
in probability {\em for a.e~realization of the environment}.

Dominated convergence, relying on the continuity of Brownian motion, then yields that
\begin{equation}\label{c14}
	\lim_{\ve\to0}\C_\ve^-=\lim_{\ve\to0}\C_\ve^+=\ga^{-1}\E(|\tilde\B_0(1)|^\a)=:\mu,
\end{equation}
and the proof is complete.

\begin{Remark}\label{muvalue}
	
	From a computation of the left hand side of the latter equality of~\eqref{c14}, we find that
	\begin{equation}\label{c15}
		\mu=\Big(1-\frac\a2\Big)\,\frac{2^{\a/2}}{\sqrt\pi}\,\,\Gamma\Big(\frac{1+\a}2\Big),
	\end{equation}
	where $\Gamma(\cdot)$ is the Gamma function.
\end{Remark}

\begin{Remark}\label{trans}
	The same approach used above to establish Theorem~\ref{main} can be used for the transient case of the environment~\footnote{Meaning specifically a spatially homogeneous birth-and-death processes
    with positive mean jump size.}, with a restriction on $\a$, and a simpler implementation, which covers
    all dimensions. We claim that, in the transient case, if $\a\in(0,1)$, then setting
    $\hxd=(\hxd_t)_{t\geq0}$, where 
	\begin{equation}\label{hxd}
		\hxd_t=\d^{\frac{1-\a}{2}}X_{\d^{-1}t},\, t\geq0,
	\end{equation}
	we have that for a.e.~realization of the environment
	\begin{equation}\label{t1} 
		\xd\Rightarrow
		\big(B_{t^{{1-\a}}}\big)_{t\geq0} \text{ as } \d\to0,
	\end{equation}
	where 
	$(B_t)_{t\geq0}$ is Brownian motion with diffusion coefficient $\ga/\hat\mu^\a$, where $\hat\mu$ is the (positive) mean jump size of the underlying birth-and-death processes (off the origin).

    The setup of Section~\ref{proofs} works mutatis mutandis as well for this case. The crucial version of 
    Proposition~\ref{lln} is easier to argue, on account of the law of large numbers for the rates seen by the random walk, which converge (strongly) to a deterministic value (namely, $\ga (\hat\mu)^{-\a}$).
We leave the fairly straightforward details to the interested reader.

\end{Remark}

%






\bibliographystyle{abbrv}
\bibliography{bibliografia}  




\end{document}